\documentclass[smallextended]{svjour3}  
\usepackage[a4paper,left=25mm,top=25mm,right=30mm,bottom=25mm]{geometry}

\usepackage{latexsym} 
\usepackage{amsmath} 
\usepackage{epsfig}
\usepackage{amssymb}
\usepackage{marvosym}
\usepackage[dvipsnames]{xcolor}
\usepackage{mathrsfs}  
\usepackage[mathcal]{euscript}
\usepackage{mathtools}
\usepackage{setspace}

\newcommand{\cS}{{\mathcal S}}

\DeclareMathOperator{\ID}{ID}

\DeclareMathOperator{\med}{med}

\begin{document}

\title{Injective split systems}

\author{M. Hellmuth \and K. T. Huber \and V. Moulton \and G. E. Scholz \and P. F. Stadler}


\institute{M. Hellmuth\at
              Department of Mathematics, Faculty of Science, Stockholm University, Sweden.\\
		   K. T. Huber \at
              School of Computer Sciences, University of East Anglia, Norwich, UK.\\
           V. Moulton \at
              School of Computer Sciences, University of East Anglia, Norwich, UK.\\
           G. E. Scholz \at
              Bioinformatics Group, Department of Computer Science \& Interdisciplinary Center for Bioinformatics, Universität
Leipzig, Germany.
			\email{guillaume@bioinf.uni-leipzig.de} \\
           P. F. Stadler \at
              {Bioinformatics Group, Department of Computer Science \& Interdisciplinary Center for Bioinformatics, Universität
              Leipzig, Germany $\ \cdot \ $
Max Planck Institute for Mathematics in the Sciences,
Leipzig, Germany $\ \cdot \ $
Department of Theoretical Chemistry, University of Vienna, Austria $\ \cdot \ $
Facultad de Ciencias, Universidad National de Colombia,
Bogot{\'a}, Colombia $\ \cdot \ $
 Santa Fe Institute, Santa Fe, NM, USA.}
              \\
}

\date{Received: date / Accepted: date}

\maketitle

\begin{abstract}
A split system $\mathcal S$ on a finite set $X$, $|X|\ge3$, is a set of 
bipartitions or splits of $X$ which contains all splits of 
the form $\{x,X-\{x\}\}$, $x \in X$. To any such split system $\mathcal S$  we can 
associate the
Buneman graph $\mathcal B(\mathcal S)$ which is 
essentially a median graph with leaf-set $X$ that displays 
the splits in $\mathcal S$.
In this paper, we consider properties of injective split systems, that
is, split systems $\mathcal S$ with the property that 
$\med_{\mathcal B(\mathcal S)}(Y) \neq \med_{\mathcal B(\mathcal S)}(Y')$ for
any 3-subsets $Y,Y'$ in $X$, where $\med_{\mathcal B(\mathcal S)}(Y)$
denotes the median in $\mathcal B(\mathcal S)$ of the three elements in $Y$ considered 
as leaves in $\mathcal B(\mathcal S)$.
In particular, we show that for any set $X$ there always exists an injective split 
system on $X$, and
we also give a characterization for when a split system is injective. We also consider
how complex the Buneman graph $\mathcal B(\mathcal S)$
needs to become in order for a split system $\mathcal S$ on $X$ to be injective. We do this 
by introducing a quantity for $|X|$ which we call the injective dimension for $|X|$, 
as well as two related quantities, called the injective 2-split and the rooted-injective
dimension. We derive some upper and lower bounds for
all three of these dimensions and also prove that some of these bounds are tight.
An underlying motivation for studying 
injective split systems is that they can be used to
obtain a natural generalization of symbolic tree maps.
An important consequence of our results is that any three-way symbolic map on $X$ can 
be represented using Buneman graphs.

\end{abstract}

\maketitle

\section{Introduction}

Let $X$ be a finite set $|X| \ge 3$.
A {\em (three-way) symbolic map (on $X$)} is a map $\delta:{X \choose 3} \to M$
to some set $M$ of symbols.
In \cite{HMS19}, a special type of symbolic map
was studied, called a {\em symbolic tree map} which arises as follows.
Let $T$ be a phylogenetic tree with leaf-set $X$ (i.e. an unrooted tree with no vertices 
of degree two and leaf set $X$ 
\cite{SS03}) 
in which each interior vertex $v$ of $T$ is labelled by some 
element $l(v)$ in $M$ by some labelling map $l$. 
The symbolic tree map $\delta$ associated to $T$ is 
the map from ${X \choose 3}$ to $M$ that is 
obtained by setting
\[
\delta(Y) = l(\med_T(Y)), \,\,\,\, Y \in {X \choose 3}, 
\]
where $\med_T(Y)$ is the unique interior vertex of $T$ 
that belongs to the shortest paths between 
each pair of the three vertices in $Y$, and ${X \choose 3}$ denotes the set of all 3-subsets of $X$. 
For example, for the
symbolic tree map $\delta$ associated to the labelled 
tree in Figure~\ref{fig-intro}(i), $\delta(\{1,2,3\})=c$, and $\delta(\{2,3,4\})=b$.
Symbolic tree maps are closely 
related to {\em symbolic ultrametrics} \cite{BD98} and also appear in 
the theory of hypergraph colourings \cite{G84} -- see \cite{HMS19} for more details, 
where amongst other results, a characterization of symbolic tree maps is presented.
There are also close connections with cograph theory \cite{H13} and modular
decompositions \cite{B22}.

\begin{figure}[h]
\begin{center}
\includegraphics[scale=0.8]{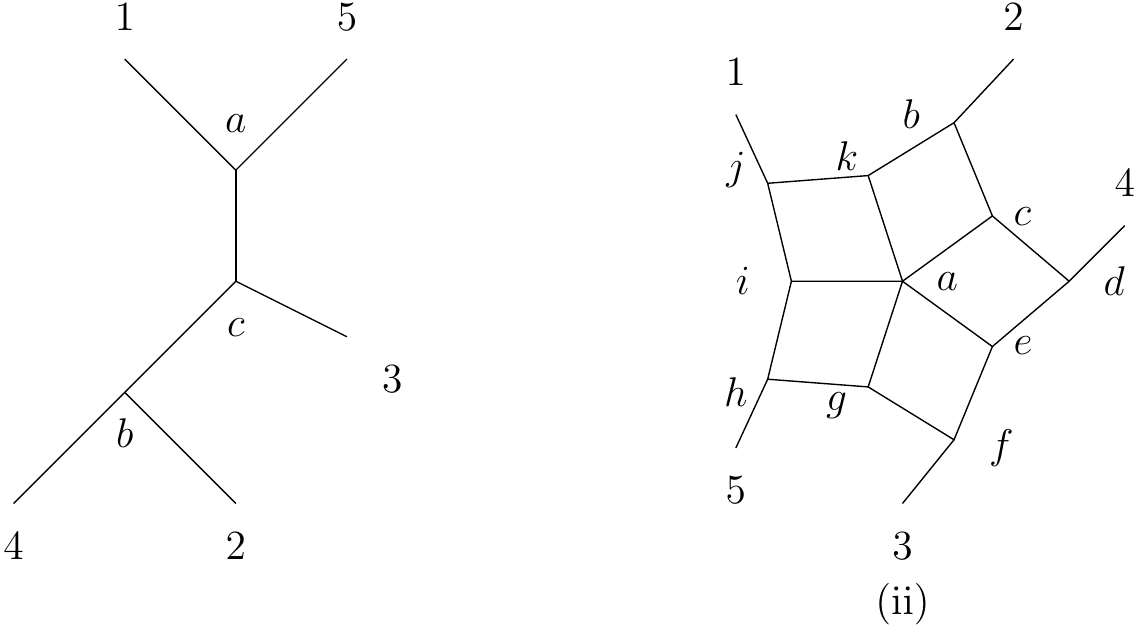}
\caption{For $X=\{1, \ldots, 5\}$, a phylogenetic tree  on $X$ in (i) and a Buneman graph 
on $X$ in (ii). In (i) the internal vertices are labelled by the 
  elements in $M=\{a,b,c\}$ and in (ii) they are labelled by the elements in 
  $M=\{a,b,\dots,k\}$.}
\label{fig-intro}
\end{center}
\end{figure}

In \cite{HMS19} it was asked how results on symbolic tree maps might
be extended to {\em Buneman graphs} \cite{D97} (see also \cite[p.8]{B22}), as these graphs
provide a natural way to generalize phylogenetic trees. More
specifically, given a {\em split system (on $X$)}, i.\,e.\, a set $\mathcal S$ of
 bipartitions or {\em splits} of $X$
that contains all splits of the form $\{\{x\}, X-\{x\}\}$, $x  \in X$, 
then the Buneman graph $\mathcal B(\mathcal S)$ on $X$ associated to $\mathcal S$ 
is essentially a median graph with leaf-set $X$
(see Section~\ref{sec:preli} for more details). 
The fact that $\mathcal B(\mathcal S)$ is a median graph implies
that for any 3-subset $Y$ of $X$, there exists a unique vertex 
$\med_{\mathcal B(\mathcal S)}(Y)$ 
in   $\mathcal B(\mathcal S)$ (or {\em median}),
that lies on shortest paths between any pair of elements in $Y$. 
Since every phylogenetic tree is a Buneman graph, the 
notion of a symbolic tree map naturally generalises by considering labelling maps $\delta$ that can be represented by 
labelling the internal vertices of some Buneman graph $\mathcal B(\mathcal S)$,
and, for any 3-subset $Y$ of $X$, taking $\delta(Y)$ 
to be the label of  $\med_{\mathcal B(\mathcal S)}(Y)$. For example, for the
map $\delta$ associated to the interior vertex-labelled
Buneman graph depicted in Figure~\ref{fig-intro}(ii), 
$\delta(\{1,2,3\})=k$, and $\delta(\{3,4,5\})=f$.

It is therefore of interest to understand under what circumstances we 
can represent for a split system $\mathcal S$ on $X$
a symbolic map $\delta$ on $X$ by labelling the vertices of some Buneman graph 
$\mathcal B(\mathcal S)$
on $X$ and vertex set $V$. In other words, we want to find some labelling map $l \colon V-X\to M$ 
such that $\delta(\{x,y,z\}) = l(\med_{\mathcal B(\mathcal S)}(Y))$ for all  
$Y \in {X \choose 3}$.
Clearly this is the case if there is some split system $\mathcal S$ on $X$ such that 
\begin{equation}\label{is-injective}
\med_{\mathcal B(\mathcal S)}(Y) \neq \med_{\mathcal B(\mathcal S)}(Y') \mbox{ for all distinct } Y, Y' \in {X \choose 3},
\end{equation}
since then we can just label the vertex $\med_{\mathcal B(\mathcal S)}(Y)$ by $\delta(Y)$ for every 3-subset $Y$ of $X$.
For example, the Buneman graph depicted in Figure~\ref{fig-intro}(ii)
enjoys Property~(\ref{is-injective}), whereas the phylogenetic tree $T$ (which is a Buneman graph
for the split system obtained by deleting all seven edges in turn) 
in Figure~\ref{fig-intro}(i) does not since, for example, $\med_T(\{1,5,3\})=\med_T(\{1,5,4\})$.
Motivated by these considerations, we call a split system $\mathcal S$ {\em injective} 
if Property~(\ref{is-injective}) holds. In this paper we shall focus on 
	understanding such split systems, in particular presenting 
	some results concerning their properties. We 
now briefly summarize them.

In the next section, we begin by presenting some preliminaries 
	concerning Buneman graphs.
In Section~\ref{sec:c2} we then prove that for any finite set $X$ with $|X|\ge 3$, 
there always exists some injective split system on $X$. In particular, we
show that the split system on $X$ which contains all 
those splits $\{A,B\}$ of $X$ with $\min\{|A|,|B|\}| \le 2$,
and the split system that is obtained by deleting any pair of
edges in a cycle with vertex set $X$ are both injective (Theorem~\ref{circbu}). 
In particular, as mentioned above, it follows that any 
symbolic map $\delta$ on a set $X$ 
can be represented by some Buneman graph.

In Section~\ref{sec:dice}, we provide a characterization of 
injective split systems (Theorem~\ref{dices}). 
This characterization is obtained by considering 
	how the restriction of a split system on $X$ to small subsets  of $X$ 
	partitions these subsets.
In particular, it implies that it can be decided if a split system $\mathcal S$ on $X$ is injective
or not by considering the restriction of $\mathcal S$  to subsets of $X$ with size at most 6. 

In general, since we can always represent a symbolic map by some
Buneman graph, we would like to find representations 
that are as simple as possible. 
Since for any split system $\mathcal S$ the Buneman graph $B(\mathcal S)$ is an 
isometric subgraph of an $|\mathcal S|$-cube in which the 
convex hull of any isometric cycle of length $k$ is a $k$-cube, $k\ge 3$, a 
natural measure for the complexity of a split system $\mathcal S$ 
is the dimension of the largest isometric $k$-cube in $B(\mathcal S)$.
We call this quantity the \emph{dimension} of $\mathcal S$; for
example, the split systems in  Figure~\ref{fig-intro}(i) and (ii) have 
dimension 1 and 2, respectively.

In Section~\ref{sec:id}, we investigate the
notion of the \emph{injective dimension}  $ID(n)$ which we define to be
the smallest dimension of any injective split system on a set 
of size $n$, $n \ge 3$. In particular, as well as giving the values of $\ID(n)$ for all $n \le 8$, we
show that $\ID(n) \leq \lfloor \frac{n}{2} \rfloor$, and that $\ID(n) \ge 3$ for all $n 
\ge 8$ (Theorem~\ref{cor:ID-numbers}).
As an immediate corollary to this result it follows that 
to represent arbitrary symbolic maps on sets $X$ of size 6 or more
using Buneman graphs, Buneman graphs that contain 3-cubes are required.

We continue by considering two variants of the injective dimension.
The first variant, $\ID_2(n)$, is considered in Section~\ref{sec:id2} and is 
given by restricting the definition of $\ID(n)$ to split systems $\mathcal S$ 
for which every split $\{A,B\} \in \mathcal S$ 
has $\min\{|A|,|B|\}| \le 2$. We show that for all $n \geq 5$, 
$\lfloor \frac{n}{2} \rfloor  \le \ID_2(n) \leq n-3$ (Theorem~\ref{id2-lb}) which implies that 
$\ID_2(5)=2$.
The second variant, $\ID^r(n)$, is considered in Section~\ref{sec:idr}, and is 
defined  by modifying the definition of injectivity as follows: 
We say that a split system $\mathcal S$ on $X$ is \emph{rooted-injective} relative to some $r \in X$ if
\[
\med_{\mathcal B(\mathcal S)}(Z\cup\{r\}) \neq \med_{\mathcal B(\mathcal S)}(Z' \cup \{r\}) \mbox{ for all distinct } Z, Z' \in {X \choose 2}. 
\]
The quantity $\ID^r(n)$ is given in an analogous way to $\ID(n)$
by taking the minimum over rooted-injective splits systems relative to $r$.
Using a recent result from \cite{B22} concerning rooted median graphs, we show 
that, in contrast to $\ID_2(n)$, 
$\ID^r(n)=2$ for all $n \ge 4$. We conclude in Section~\ref{sec:discuss} with a 
discussion of some open problems.

\section{Preliminaries} \label{sec:preli}

\subsection{Graphs and median graphs}	\label{sec:graphs}
    
    We consider undirected graphs $G=(V,E)$ 
    whose vertex sets $V$ are finite with $|V| \ge 2$, and whose
	edge sets $E$ are contained in $\binom{V}{2}$, i.e., graphs without loops and multiple
	edges.  A \emph{leaf} in such a graph is a vertex with degree one.  
	A \emph{cycle} is a connected graph in which every vertex has degree two.
	The \emph{length} of a cycle $C$ is the number of edges or, equivalently, the
	number of vertices in $C$.  
	A connected graph that does not contain a cycle is called a \emph{tree}.

	If $G$ is connected then we denote by $d_{G}(v,w)$  the length of a shortest path between two
	vertices $v$ and $w$ of $G$. Note that $d_{G}(v,w)=0$ if and only if $v=w$.
	A connected subgraph $G'$ of $G$ is called isometric if
   $d_{G'}(v,w) = d_{G}(v,w)$, for all vertices $v$ and $w$ in $G'$.
	A vertex $x$ in $G$ is called a \emph{median} of three vertices $u,v,w\in V$
	if $d_{G}(u,x)+d_{G}(x,v)=d_{G}(u,v)$,
	$d_{G}(v,x)+d_{G}(x,w)=d_{G}(v,w)$ and
	$d_{G}(u,x)+d_{G}(x,w)=d_{G}(u,w)$.  A connected graph is called a
	\emph{median graph} if any three of its vertices have a unique median \cite{Mulder1978}.
	In other words, $G$ is a median graph if for all vertices $u$, $v$, and $w$ in $G$, there is
	a unique vertex that belongs to shortest paths between each pair of $u, v$
	and $w$.  We denote the unique median of three vertices $u$, $v$ and $w$ in
	a median graph $G$ by $\med_{G}(u,v,w)$. Median graphs
	have several interesting characterizations and properties, see e.g. \cite{mulder2011median}.
	For example, a connected graph $G$ is a median graph if and only if the 
	convex hull\footnote{A subset $G'$ of a graph $G$ 
		is \emph{convex} if for any two vertices $v,w$ in $G'$
		\emph{every} shortest path between $v$ and $w$ is a subgraph of $G'$.}
	 of any isometric cycle of $G$ is a hypercube (see e.g. \cite{klavzar1999median}).
	
\subsection{Buneman graphs}

From now on, we let $X$ be a finite set with $|X|\ge 3$. A {\em split (of 
$X$)} is a bipartition $A|B=B|A$ of $X$ into
two non-empty subsets, that is, $A,B\subset X$, $A\cap B=\emptyset$ and $A\cup 
B=X$. 
For simplicity, we write $a_1\ldots a_k|b_1\ldots b_l$ or $a_1\ldots a_k|\overline{a_1\ldots a_k}$ for a split
$A|B$ if $A=\{a_1,\ldots, a_k\}$ and $B=\{b_1,\ldots, b_l\}$, for some $k,l\geq 1$.
We call the sets $A$ and $B$ the \emph{parts} of the split $A|B$. 
If $S=A|B$ 
is such that $|A|<|B|$ then we call $A$ the {\em small part} of $S$.
The {\em size} of a split $A|B$ is defined as $\min\{|A|,|B|\}$,
and if a split $S$ has size $r$ we call $S$ an {\em $r$-split}.
A split $A|B$ of $X$ is called \emph{trivial} if it has size $1$ or, equivalently, 
if $A|B$ is of the form $x|\overline{x}$ for some $x\in X$. 
For a split $S=A|B$ of $X$, we let $S(x)$ denote the part of 
$S$ that contains $x$. 
We say that
$S$ {\em separates} two elements $x$ and $y$ in $X$ if $S(x)\not=S(y)$.
From now on we shall assume that all split systems on $X$  contain all trivial splits on $X$.

Following \cite{D97}, we define for a split system $\mathcal S$ on $X$, the 
{\em Buneman graph $\mathcal B(\mathcal S)$ (on $X$)} 
to be the graph with
vertex set consisting of all maps $\phi: \mathcal S \to \mathcal P(X)$ 
satisfying the following two conditions:
\begin{itemize}
\item[(B1)] For all $S\in \mathcal S$, $\phi(S) \in S$.
\item[(B2)] For all $S, S' \in \mathcal S$ distinct, $\phi(S) \cap \phi(S') \neq 
\emptyset$.
\end{itemize}
Two vertices $\phi$ and $\phi'$ in $\mathcal B(\mathcal S)$ are joined by an edge if 
there 
is a unique split $S \in \mathcal S$ such that $\phi(S) \neq \phi'(S)$. 
For example, the graphs in Figure~\ref{fig-intro}(i) and (ii) 
are Buneman graphs on $X=\{1,\ldots, 5\}$ for the split systems
\[\mathcal S_1=\{15|234,24|135\} \cup \{x|\overline{x} \,:\, x \in X\}\]
and
\[\mathcal S_2=\{15|234,24|135, 12|345, 34|125,35|124\} \cup \{x|\overline{x} \,:\, x \in X\},\] 
respectively.

We now summarise some relevant properties of the Buneman graph
(for proofs of these facts see e.g. \cite[Chapter 4]{D12}; see also \cite{BG91} using different notation).

	\begin{enumerate}
		\item[(S1)]  For all $x \in X$, the map 
		$\phi_x: \mathcal S \to \mathcal P(X)$ given by
		putting $\phi_x(S)=S(x)$, for all $S \in \mathcal S$, is a leaf 
		in $\mathcal B(\mathcal S)$. 
		\item[(S2)]  Let $S=A|B\in \mathcal{S}$. Then 
									the removal of all edges $\{\phi,\phi'\}$ in
									$\mathcal B(\mathcal S)$ with $\phi(S)\neq \phi'(S)$
								   disconnects 
								   $\mathcal B(\mathcal S)$  into precisely two  connected 
								   components, one of which contains the leaves $\phi_a$, $a \in A$ and the other
								   the leaves $\phi_b$, $b \in B$.
		\item[(S3)] $\mathcal B(\mathcal S)$ is a median graph.
		\item[(S4)] $\mathcal B(\mathcal S)$ is an isometric subgraph of 
			the $|\mathcal S|$-dimensional hypercube consisting of all those maps $\phi: \mathcal S \to \mathcal P(X)$ 
			that only satisfy Property~(B1) in the definition of the Buneman graph 
			(with edge set defined in the analogous same way).
		\item[(S5)] For any three vertices $\phi_1, \phi_2,\phi_3$
	    in $\mathcal B(\mathcal S)$, the median 
		of $\phi_1, \phi_2$ and $\phi_3$ in $\mathcal B(\mathcal S)$ is the map that
		assigns to each split $S \in \mathcal S$ the part of $S$ of multiplicity two or more in the 
		multiset $\{\phi_1(S),\phi_2(S),\phi_3(S)\}$ (see also \cite[p.\ 1905, Equ.\ (1)]{D11}).
	\end{enumerate}

Suppose that $\cS$ is a split system on $X$. 
In light of Property~(S1), we shall consider $X$ as being the leaf-set of $\mathcal B(\mathcal S)$, since
each $x \in X$ corresponds to the map $\phi_x$ in $\mathcal B(\mathcal S)$.
As an example for (S2), consider the  tree in Figure~\ref{fig-intro}(i). 
Removing the edge associated to the split $15|234$ disconnects the tree 
into two trees with  leaf sets $\{1,5\}$ and $\{2,3,4\}$, respectively. 
In this way, we see that $\mathcal B(\mathcal S_1)$ displays 
each of the splits in $\mathcal S_1$.

Note that by Property~(S3) and the fact mentioned at the end Section~\ref{sec:graphs},
 the convex hull of any isometric  cycle in $\mathcal B(\mathcal S)$ is a hypercube. In light of this, we
define the {\em dimension} $\dim(\mathcal S)$ of a split system $\mathcal S$ 
to be the dimension of the largest hypercube contained in $\mathcal B(\mathcal S)$
in case $\mathcal B(\mathcal S)$ is not a phylogenetic tree and one otherwise.
This dimension can be characterized in terms of splits as follows.
Suppose $S=A|B$ and $T=C|D$ are two splits in $\mathcal S$. 
Then $S$ and $T$ are called {\em incompatible} if $S\not=T$ and 
$A\cap C$, $A \cap D$, $B\cap C$ and $B \cap D$ are all non-empty; otherwise $S$ 
and $T$ are called {\em compatible}. Calling a set $\mathcal S$ of splits {\em incompatible}
if any two splits in $\mathcal S$ are incompatible, then $\dim(\mathcal S)$  
is equal  to the maximum size of an  incompatible subset  
of $\mathcal S$
(see e.g. \cite[p. 445]{C08}). If $\mathcal B(\mathcal S)$ contains a cycle then 
it must contain a hypercube of dimension two or more. 
Hence, a split system $\mathcal S$ on $X$
is 1-dimensional 
if and only if $\mathcal B(\mathcal S)$ is a phylogenetic tree on $X$ (in which case it has  $|\mathcal S|+1$ vertices and $|X|$ leaves), a 
fact which also holds if and only if every 
pair of splits in $\mathcal S$ is compatible (see e.g. \cite{D97}). 
In particular, as mentioned in the introduction, it follows that 
any phylogenetic tree is a Buneman graph of
some split system, and that any two distinct splits in this split system must be compatible.

\section{Two families of injective split systems}\label{sec:c2}

Let $\mathcal S$ be a split system on $X$.
For $Y =\{x,y,z\} \in {X \choose 3}$, we let $\phi_Y = \phi_{xyz}=med_{\mathcal B(\mathcal S)}(Y)$ denote
the median of $\phi_x,\phi_y,\phi_z$ in $\mathcal B(\cS)$, which exists by Property (S3).
In this notation, $\mathcal S$ is
{\em injective}  if for all $Y,Y'  \in {X \choose 3}$ distinct, we have
$\phi_Y \neq \phi_{Y'}$.
Note that if $|X|=3$, then there is only one split system $\mathcal S$ on $X$ 
(the one that  contains only trivial splits), and that $\mathcal S$  is injective, since $|\binom{X}{3}|=1$.
In this section, we show that for every set  $X$ with $|X|\ge 4$ there
exists an injective split system on $X$. To do this, we shall present two infinite 
families of injective split systems.

We begin with a simple but useful lemma.

\begin{lemma}\label{medinbu}
	Let $\mathcal S$ be a split system on a set $X$, $|X| \geq 3$, and 
	let $x,y,z \in X$ distinct. Then $\phi_{xyz}$ is the (unique) map
in $\mathcal B(\mathcal S)$ that assigns to each 
	split 
	$S \in 
	\mathcal S$ the part $A\in S$ for which
	$|A \cap \{x,y,z\}| \geq 2$.
\end{lemma}
\begin{proof}
	Let $S\in\mathcal S$. Then  $\phi_v(S)=S(v)$, for all $v\in \{x,y,z\}$. By  Property~(S5), 
	$\phi_{xyz}(S)$ is the part of $S$
	that appears twice (or more) in the multiset $\{S(x),S(y),S(z)\}$, that is, the part of $S$ 
	that contains (at least) two elements of $\{x,y,z\}$.
\end{proof}

Now, a split system $\mathcal S$ on $X$ is called {\em circular} \cite{BD92} if 
there exists a labelling $x_1, \ldots, x_n$, $n=|X|$, of the elements of 
$X$ such that all splits of $\mathcal S$ are of the form 
$x_i x_{i+1} \ldots x_j|\overline{x_i x_{i+1} \ldots x_j}$, some $1 
\leq i \leq j \leq n$.
If $\mathcal S$ is a circular split system on $X$ and there is no circular split system $\mathcal S'$ on $X$
such that $\mathcal S \subsetneq \mathcal S'$, then we say that $\mathcal S$ is a \emph{maximal circular}
split system on $X$. Note that a maximal circular split system on $X$
has size ${|X| \choose 2}$ \cite[Section 3]{BD92}.

We now use Lemma~\ref{medinbu} to show that there exist
families of split systems that are injective. 

\begin{theorem}\label{circbu}
Let $\mathcal S$ be a split system on $X$, $|X| \geq 4$. Then:
\begin{itemize}
\item[(i)] If $\mathcal S$ contains all 2-splits of $X$,  then $\mathcal 
					 S$ is injective.
\item[(ii)] If $\mathcal S$ is maximal circular, then $\mathcal S$ is injective.
\end{itemize}
\end{theorem}
\begin{proof}
For both (i) and (ii), let $Y=\{x,y,z\}$ and $Y'$ denote two distinct subsets of $X$ of size 
$3$. Assume without loss of generality that $x\notin Y'$.

(i) By Lemma \ref{medinbu}, $\phi_Y$ 
is the unique map $\mathcal B(\mathcal S)$ that assigns to each
		split $S \in 
		\mathcal S$ the part $A$ of $S$ such that $|A \cap Y| \geq 2$.
		 It follows that for $S=xy|\overline{xy}$ (which is an element of $\mathcal S$ as it has size two), $\phi_{xyz}(S)=\{x,y\}$.
		Since $x\not\in Y'$, we obtain $\phi_{xyz}(S)= X-\{x,y\}$.
	Consequently, $\phi_Y \neq \phi_{Y'}$.

(ii) 
Put $X=\{x_1,\ldots, x_n\}$, $n\geq 4$.
Then there exist $i,j,k\in\{1,\ldots, n\}$ with $i<j<k$ ($\mathrm{mod}\ n$) such that $x=x_i$, $y=x_j$ and $z=x_k$.
With respect to the circular ordering of $X$ induced by $\mathcal S$ it follows that
one of the four sets $\{x=x_i,x_{i+1}, \ldots, y=x_l\}$, $\{y=x_l,x_{l+1}, \ldots, x=x_i\}$, 
$\{x=x_i,x_{i+1}, \ldots, z=x_k\}$ and $\{z=x_k,x_{i+1},\ldots, x=x_i\}$
must contain at most one element of $Y'$.
Let $A$ be such a set. Since $\mathcal S$ is maximal circular by assumption, it follows that the split $S=A | X-A$ 
is contained in $\mathcal S$. By Lemma \ref{medinbu}, $\phi_Y(S)=A\not=X-A=\phi_{Y'}(S)$.
Hence, $\phi_Y \neq \phi_{Y'}$.
\end{proof}

\begin{figure}[h]
\begin{center}
\includegraphics[scale=0.5]{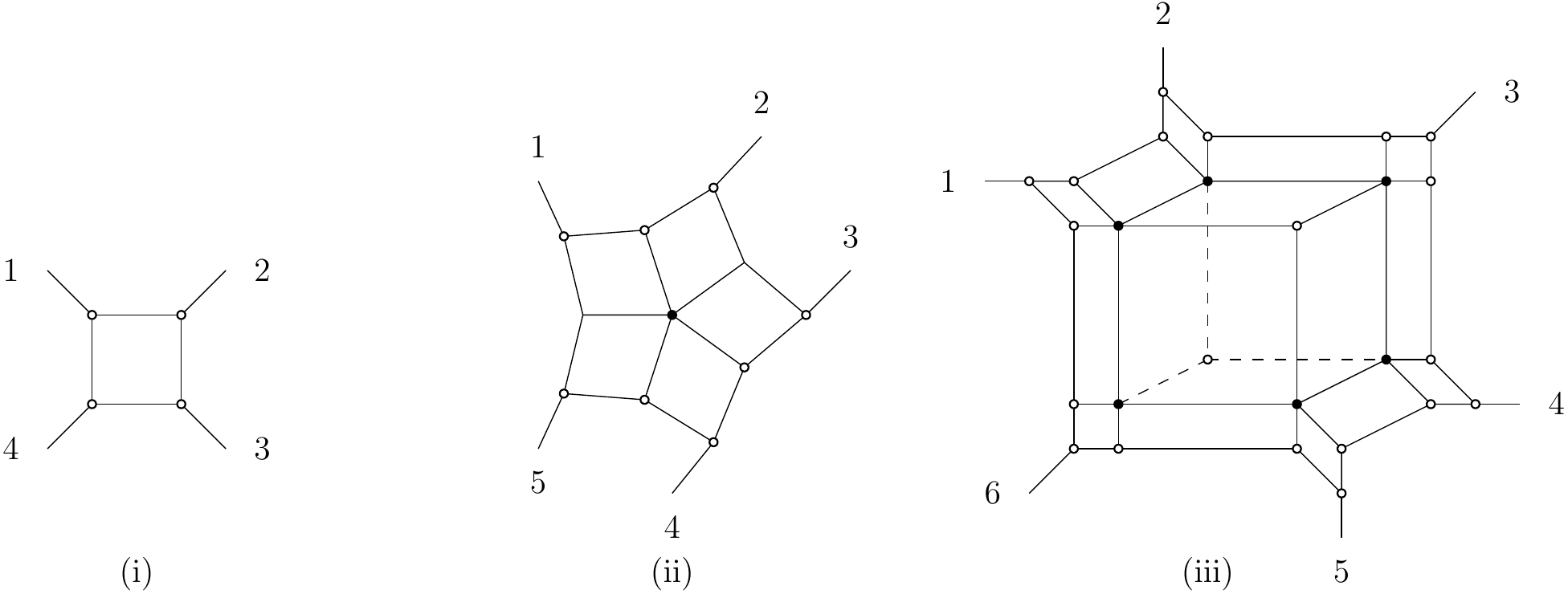}
\caption{For $X=\{1,\dots,n\}$ with $n=4, 5,6$ and the induced natural ordering of $X$, the
respective Buneman graphs on $X$ of the associated maximal circular 
	split systems on $X$ where (i) is $n=4$, (ii) is $n=5$, and (iii) is $n=6$. In all cases, 
	leaves are indicated in terms of the elements in $X$. Vertices that are of the form 
	$\phi_{xyz}$, some $x,y,z \in X$, are indicated as unfilled circles 
	and all other non-leaf vertices are indicated as filled circles.}
\label{bubu}
\end{center}
\end{figure}

In view of Theorem~\ref{circbu}~(ii), it is interesting to understand if 
maximal circular split systems admit proper subsets that are also injective. As it turns out, 
the answer is no in general, as we show in our next result. 

\begin{proposition}\label{prop:circ-proper-subset}
Let $\mathcal S$ be a circular split system  on $X$ with $|X|\geq 4$ and 
let $\mathcal S'$ denote a split system on $X$ that is contained in 
$\mathcal S$ as a proper subset.
Then $\mathcal S'$ is not injective.
\end{proposition}
\begin{proof}
Let $S_0$ be a non-trivial split in $\mathcal S-\mathcal S'$. 
	  We show that there exists 
		two subsets $Y$ and $Z$ of $X=\{1,\ldots, n\}$ distinct such that $\phi_Y(S)=\phi_Z(S)$ 
	  for all $S \in \mathcal S-\{S_0\}$.
		In particular, $\phi_Y(S)=\phi_Z(S)$ for all $S \in \mathcal S'$, so $\mathcal S'$ is not injective.

Assume that $\mathcal S$ is circular for the
natural ordering of $X$. Without loss of generality, we may assume that
$S_0=1\ldots k|k+1 \ldots n$,  some $2 \leq k \leq \frac{n}{2}$. Consider the sets $Y=\{n,1,k\}$ and $Z=\{n,1,k+1\}$. 
Let $S\in \mathcal S-\{S_0\}$. If $S(n)=S(1)$ then, by Lemma \ref{medinbu},
$\phi_Y(S)=S(n)=\phi_Z(S)$. If $S(n)\not=S(1)$ then $S$ must be of
the form $1\dots\ell | \ell+1\dots n$, some $1\leq \ell\leq n-1$. Since $S\neq S_0$, we have $\ell\neq k$.
Hence, $S(k)=S(k+1)$. Moreover, since
$S(1)\not=S(n)$ either $1$ or $n$ must 
be contained in $S(k)$. We can then apply Lemma \ref{medinbu} again
to conclude that $\phi_Y(S)=\phi_Z(S)$ which completes 
the proof.
\end{proof}

We remark that a similar result to Proposition~\ref{prop:circ-proper-subset}
does not necessarily hold for non-circular split systems even if they are 
injective. For example, Theorem~\ref{circbu}(i) implies that
the split system $\mathcal S$ on $X=\{1, \ldots, n\}$, $n \geq 5$, that 
consists precisely of all trivial splits and 2-splits on $X$ is injective. Let
$\mathcal S^*$ denote the split system containing all splits of $\mathcal S$ except those of the form 
$1x|\overline{1x}$, $x \in X-\{1\}$. Then, $\mathcal S^*$ is injective. To see this, 
consider the proof of Theorem~\ref{circbu}(i). Then, up to 
potentially having to relabel the elements of $Y$ and $Y'$,
the elements $x$ and $y$ can always be chosen to be 
different from $1$. Hence, the split $S=xy|\overline{xy}$ such that 
$\phi_Y(S) \neq \phi_{Y'}(S)$ can always  be chosen in such a way that $ S\in\mathcal S^*$. 
As a consequence, it follows that for all $Y, Y' \in {X \choose 3}$ distinct, 
there exists a split $S$ of $\mathcal S^*$ such that $\phi_Y(S) \neq \phi_{Y'}(S)$ 
which implies that $\mathcal S^*$ is injective.

\section{Characterization of injective split systems and Dicing}\label{sec:dice}
 
In this section, we characterize injective split systems
(Theorem~\ref{dices}).  To this end, we shall consider the restriction of a
split system on $X$ to subsets of $X$ which is defined as follows.  Given a
split system $\mathcal S$ on $X$, and a subset $Y \subseteq X$ with
$|Y|\geq 3$ then we define the restriction $\mathcal S|_Y$ of $\mathcal S$
to $Y$ as the set of splits $S|_Y$ restricted to $Y$, that is,
\[
\mathcal S|_Y = \{ S|_Y= A\cap Y| B \cap Y \,:\, A|B \in \mathcal S 
\}.
\] 
Note that $\mathcal S|_Y$ is in fact a split system on $Y$ since $\mathcal S|_Y$
contains all trivial splits on $Y$.
We begin by proving a useful lemma concerning such restrictions.

\begin{lemma}\label{cases}
	Suppose that $S$ is a split on $X$ with $|X|\geq 4$, and 
	that $x,y,z,p$ are distinct elements of $X$. Then the following holds for 
	$Y=\{x,y,z,p\}$.
	\begin{enumerate} 
	\item[(i)] $\phi_{xyz}(S) \neq \phi_{xyp}(S)$ if and only if 
	$S|_Y\in \{xz|yp,yz|xp\}$. In particular, $S|_Y\not=xy|pz$.
	\item[(ii)] If $|X|\geq 5$ and $q\in X-Y$ then $\phi_{xyz}(S) \neq \phi_{xpq}(S)$ if and only if 
	$S|_{Y\cup\{q\}}$ is one of the splits $yz|xpq$,
	$pq|xyz$, $xy|zpq$, $xz|ypq$, $xp|yzq$ or $xq|yzp$.
	\item[(iii)]  If $|X|\geq 6$ and $q,r\in X-Y$ distinct then $\phi_{xyz}(S) \neq \phi_{pqr}(S)$ if and only if 
	$S|_{Y\cup\{q,r\}}$ is a 3-split or it is a 2-split of $Y$ 
	whose part of size 2 is contained in $\{x,y,z\}$ or $\{p,q,r\}$.
	\end{enumerate}
\end{lemma}
\begin{proof} 
	To see Assertion~(i) observe that,	
	by Lemma~\ref{medinbu}, we have $\phi_{xyz}(S) \neq
\phi_{xyp}(S)$ if and only if one of $A$ and $B$, say $A$, contains at
	least two elements of $\{x,y,z\}$ while  $B$ contains at
	least two elements of $\{x,y,p\}$. Since $A\cap B=\emptyset$, this is only
	possible if and only if $z\in A$ and $p\in B$ while either $x\in A$ and $y\in B$ or
	$y\in A$ and $x\in B$. The latter is equivalent to  $S|_Y\in
	\{xz|yp,yz|xp\}$ which, in particular, implies that $S|_Y\neq xy|pz$.
	Hence, Assertion~(i) must hold.

	To see Assertion~(ii), observe that, by
	Lemma~\ref{medinbu}, $\phi_{xyz}(S) \neq
	\phi_{xpq}(S)$ if and only if one of $A$ and $B$,
		say $A$, contains at least two elements of $\{x,y,z\}$ and $B$  contains
	at least two elements of $\{x,p,q\}$. As is easy to see, this is the case if
	and only if $S|_{Y'}$ is not a trivial split on $Y'=Y\cup\{q\}$ and one of
	$S(y)=S(z)$ or $S(p)=S(q)$ holds. Consideration
		of all ten non-trivial splits on $Y'$ shows that $S|_{Y'}$ must be one of $yz|xpq$, $pq|xyz$, $xy|zpq$,
	$xz|ypq$, $xp|yzq$ or $xq|yzp$. Hence,
		Assertion~(ii) must hold.
	
	To see Assertion~(iii),  observe that, by
	Lemma~\ref{medinbu}, $\phi_{xyz}(S) \neq
	\phi_{pqr}(S)$ holds if and only if one of $A$ and $B$, say $A$, contains at least two
	elements of $\{x,y,z\}$ and $B$ contains at least two elements of
	$\{p,q,r\}$. Put $Y'=Y\cup\{q,r\}$, $A'= A\cap Y'$ and
	$B'=B\cap Y'$. Since $A\cap B=\emptyset$ it follows
	that $S|_{Y'}$ must be a 2- or 3-split and that if $S|_{Y'}$ is a 2-split, 	its part
	 of size 2 is contained in $\{x,y,z\}$ or
	$\{p,q,r\}$. 
	
	Conversely, put $A=\{x,y,z\}$ and $B=\{p,q,r\}$ again. If  
	$S|_{Y'} = A'|B'$ is a 3-split on $Y'=Y\cup\{q,r\}$ then, clearly,
	$|A'\cap \{x,y,z\}|\geq 2$ and $|B'\cap \{p,q,r\}|\geq 2$. Since
	$A'\subseteq A$ and $B'\subseteq B$, we obtain $\phi_{xyz}(S) \neq
	\phi_{pqr}(S)$. Furthermore, if $S|_{Y} = A'|B'$ is a 2-split such that the
	part 	of size 2 is 	contained in $A$ or $B$, then the
	other part must be of size 4 and must contain $B$ or $A$.
	Consequently, $\phi_{xyz}(S) \neq \phi_{pqr}(S)$. Hence, Assertion~(iii) must
	hold.
\end{proof}

We now make a key definition. We shall say that a split
system $\mathcal S$ on $X$

\begin{description}
\item[$\bullet$ \textnormal{\em $4$-dices $X$}] if
$|X| < 4$ or for all $Y \in {X \choose 4}$,
$\mathcal S|_Y$ contains at least two 2-splits,

\item[$\bullet$ \textnormal{\em 5-dices $X$}] if 
$|X| < 5$ or  for all $Y \in {X \choose 5}$,
$\mathcal S|_Y$ contains at least five 2-splits, and

\item[$\bullet$ \textnormal{\em 6-dices $X$}] if 
$|X| < 6$ or for all $Y \in {X \choose 6}$, $\mathcal 
S|_Y$
contains at least one 3-split or a \emph{triangle of 2-splits}, that is, 
three 2-splits of the form $xy|Y-\{x,y\}$, $xz|Y-\{x,z\}$ and 
$yz|Y-\{y,z\}$ where $x$, $y$, and $z$ are distinct elements in $Y$.

\end{description}

%
%

Note that, in general, if a split system on $X$ $k$-dices $X$  it need not $k'$-dice $X$, for
$k,k'\in \{4,5,6\}$ distinct. 
	Nevertheless, some interesting relationship between these concepts hold as the next lemma illustrates.

\begin{lemma}\label{lm-up}
Suppose $\mathcal S$ is a split system on $X$.
\begin{itemize}
\item[(i)] If $\mathcal S$ 4-dices $X$ and $|X| \geq 5$ then, for all $Y 
\in {X \choose 5}$, 
$\mathcal S|_Y$ contains at least four 2-splits.
\item[(ii)] If $\mathcal S$ 5-dices $X$ and $|X| \geq 6$ then, for all $Y 
\in {X \choose 6}$, 
$\mathcal S|_Y$ contains a 3-split or(at least) eight 
2-splits.
\end{itemize}
\end{lemma}

\begin{proof}
(i) Suppose that $\mathcal S$ 4-dices $X$ and that $|X|\geq 5$.
Let $Y=\{x,y,z,t,u\}\in {X \choose 5}$ and $Y'=\{x,y,z,t\}\in {X \choose 
4}$.
Since $\mathcal S$ 4-dices $X$, $\mathcal S|_{Y'}$ contains at least two 
2-splits $S'_1$ and $S'_2$. Hence, $\mathcal S|_Y$ contains two 
splits $S_1$ and $S_2$ such that
$S_1|_{Y'}=S'_1$ and $S_2|_{Y'}=S'_2$. 
Moreover, since $S'_1$ and $S'_2$ are both 2-splits on $Y'$, the part 
$A_1$ of $S_1$ and 
$A_2$ of $S_2$ of size 2 does
not contain $u$. Note that $A_1$ and $A_2$ must be parts of $S'_1$ and $S'_2$, 
respectively. In particular, since $S'_1$ and $S'_2$ are splits on  
$Y'$ and $S'_1\neq S'_2$ it follows that $|A_1 \cap A_2|=1$. Without loss of 
generality, 
we may assume that $A_1 \cap A_2=\{x\}$.
Replacing $Y'$ by $Y''=\{y,z,t,u\}$ and using an analogous argument implies that  $\mathcal 
S|_Y$ also contains two distinct 2-splits on $Y$, call them $S_3$ and $S_4$, whose 
parts of size $2$ do not contain $x$. In particular, 
$S_3$ and $S_4$ are distinct from $S_1$ and $S_2$. 
In summary, $\mathcal S|_Y$ contains at least four distinct 2-splits.

(ii)  Suppose that $\mathcal S$ 5-dices $X$ and that $|X|\geq 6$.
Let $Y \in {X \choose 6}$. If  $\mathcal S|_Y$ contains a 3-split
we are done. Hence, assume $\mathcal S|_Y$ does not contain a 3-split.
Since $|Y|=6$ it follows  that a split in $\mathcal S|_Y$ must be trivial or a
 2-split. We continue with showing that $\mathcal S|_Y$
contains at least eight 2-splits. 
Let $x\in Y$. Since a split in $\mathcal S|_Y$ is either trivial  or a
	2-splits, all 2-splits of $\mathcal S|_{Y-\{x\}}$ 
correspond to the 2-splits of $\mathcal S|_Y$ whose
small part does not contain $x$. 
We claim that there exists an element $x_0$ of $Y-\{x\}$ that belongs to the small
part of at least three 2-splits of $Y$.
To see this, we consider the following two cases: (a)
 $\mathcal S|_Y$ does not contain a split whose small part contains $x$
  and (b)  $\mathcal S|_Y$ contains a split whose small part contains 
  $x$.

In case of (a), let $Y' =
	\{x,y,a,b,c\}$ be a subset of $Y$ of size $5$.
	Since $\mathcal S$ 5-dices $X$, it follows that $\mathcal S|_{Y'}$ contains at least five
	 of the ${4\choose 2}=6$ possible 2-splits in
	$\{ya|\overline{ya}, yb|\overline{yb}, yc|\overline{yc}, ab|\overline{ab}, ac|\overline{ac},$ 
	$bc|\overline{bc}\}$
	 that might be contained in
	$\mathcal S|_Y$ and do not have $x$ in their small part. It is now straight-forward to
	verify that there is some $x_0 \in
	Y'-\{x\}$ such that $\mathcal S|_{Y'}$ contains three 2-splits whose small
	part contains $x_0$.

Consider now Case (b). Since
	$\mathcal S$ 5-dices $X$ and $|X|\geq 6$, $\mathcal S|_{Y-\{x\}}$ contains 
	again at least five 
	2-splits. Then if there exists an element $x_0 \in
	Y-\{x\}$ such that $\mathcal S|_{Y-\{x\}}$ contains three 2-splits whose
	small part contains $x_0$ then the claim follows. If
	this is not the case, then consideration of all ${5\choose 2}=10$ possible 2-splits in
		$\mathcal S|_{Y-\{x\}}$ shows that $\mathcal S|_{Y-\{x\}}$ must contain exactly five
	2-splits and that all elements of $Y-\{x\}$ must belong to the small part of exactly two
	2-splits of $\mathcal S|_{Y-\{x\}}$. In addition, by assumption on $x$,
	there exists an element $x_0$ of $Y-\{x\}$ such that $\{x,x_0\}$ is the small
	part of a split of $\mathcal S|_Y$. Since $x_0$ also belongs to the small part of
	exactly two 2-splits in $\mathcal S|_{Y-\{x\}}$, it follows that $x_0$
	belongs to the small part of exactly three 2-splits of $\mathcal
	S|_{Y-\{x\}}$. 
	Hence, there is some $x_0 \in
	Y'-\{x\}$ such that $\mathcal S_{Y'}$ contains three 2-splits whose small
	part contains $x_0$.

In summary, in both Case~(a) and (b), there is some $x_0 \in Y'-\{x\}$ such that $\mathcal 
S|_{Y'}$ contains three 2-splits whose small part contains $x_0$.
Moreover,  $\mathcal S|_{Y-\{x_0\}}$ contains at least five 2-splits
because $\mathcal S$ 5-dices $X$ and $|Y|=6$.
Since the small part of a split in $\mathcal S|_{Y-\{x_0\}}$ is also the small part of a split in 
$\mathcal S|_Y$ whose small part does not contain $x_0$, it follows that there also exists at least five 
2-splits in $\mathcal S|_Y$ whose small part does not contain $x_0$. 
Hence, 
$\mathcal S|_Y$ contains at least eight 2-splits.
\end{proof}

To prove the main theorem of this section, we require a further result 
concerning dicing.

\begin{proposition}\label{prop45d}
Suppose $\mathcal S$ is a split system on $X$ with $|X|\geq 4$. Then the following holds.
\begin{itemize}
\item[(i)] $\mathcal S$ 4-dices $X$ if and only if for all
$A,B \in {X \choose 3}$ 
with $|A \cap B|=2$, we have $\phi_A \neq \phi_B$. 
\item[(ii)] If $|X|\geq 5$ then $\mathcal S$ 4- and 5-dices $X$ if and only if for all 
distinct $A,B \in {X \choose 3}$ 
with $A \cap B \neq \emptyset$, we have $\phi_A \neq \phi_B$.
\end{itemize}
\end{proposition}

\begin{proof}
(i) Let $A=\{x,y,z\}$ and $B=\{x,y,t\}$ be subsets of $X$ and let $Y=A \cup B$.
Assume first that $\mathcal S$ 4-dices $X$. Then $\mathcal S|_Y$ contains at
least two 2-splits because $|Y|=4$. In particular, $\mathcal S|_Y$ contains at least one 2-split
$S$ distinct from $xy|tz$. By Lemma~\ref{cases}(i), it follows that 
$\phi_A(S) \neq \phi_B(S)$. Consequently, $\phi_A \neq \phi_B$.

Conversely, if $\phi_A \neq \phi_B$, then there exists a split 
$S$ in $\mathcal S$ such that $\phi_A(S) \neq \phi_B(S)$. By Lemma~\ref{cases}(i), 
$S|_Y \in  \{xz|yt,yz|xt\}$. If $S|_Y = xz|yt$, then consider the set  
$C=\{x,z,t\}$.
Since, by assumption, $\phi_A \neq \phi_C$ there must exist a split $S'$ in 
$\mathcal S$ such $\phi_A(S') \neq \phi_C(S')$.
By Lemma~\ref{cases}(i) it follows that $S'|_Y \neq S|_Y$. If $S|_Y = yz|xt$ then an
analogous argument with $C$ replaced by  $D=\{y,z,t\}$ implies that there exists 
a split $S''$ with $S''|_Y \in  \{yx|zt,yt|zx\}$. Lemma~\ref{cases}(i) implies again that
$S|_Y \neq S''|_Y$. Hence, 
$\mathcal S|_Y$ contains at least two 2-splits one of which is $S|_Y$ and 
the other is $S'|_Y$ or $S''|_Y$.

(ii) Assume first that $\mathcal S$ 4-dices and 5-dices $X$. 
Let $A,B\in \binom{X}{3}$ distinct such that $A\cap B\not=\emptyset$. If $|A\cap B|=2$ then, by 
		Proposition~\ref{prop45d}(i), $\phi_A \neq 
		\phi_B$ must hold. So assume that $|A\cap B|\not=2$.
		Let $A=\{x,y,z\}$ and $B=\{x,p,q\}$. Then $|A\cap B|=1$. 
		Since $\mathcal S$ 5-dices $X$ and $|X|\geq 5$ it follows that $\mathcal S|_Y$ 
contains at least five 2-splits where $Y=A\cup B$. Since there are exactly ${10\choose 2}$ 2-splits on $Y$, 
it follows that $\mathcal S|_Y$ contains at least one of the six 2-splits
in $\{yz|xpq, pq|xyz, xy|zpq, xz|ypq, xp|yzq,xq|yzp\}$. 
By Lemma~\ref{cases}(ii), it follows that $\phi_A \neq \phi_B$.

Conversely, assume that for all distinct $A,B \in {X \choose 3}$ 
with $A \cap B \neq \emptyset$  we have  that $\phi_A \neq \phi_B$.
If $|A\cap B|=2$ then $\mathcal S$ 4-dices $X$ in view of Proposition~\ref{prop45d}(i).
To 
see that $\mathcal S$ also 5-dices $X$, we need to show in view of $|X|\geq 5$ that for all $Y\in \binom{X}{5}$ the split system $\mathcal S|_Y$ contains at 
least five 2-splits. 

Let $Y\in \binom{X}{5}$.
Since $\mathcal S$ 4-dices $X$, 
it follows by Lemma~\ref{lm-up}(i) that $\mathcal S|_Y$ contains at least four 2-splits. 
Assume  for contradiction that $\mathcal S|_Y$ contains 
precisely four 2-splits $S_1, \ldots, S_4$. For all $1\leq i\leq 4$, let $A_i$ denote the small part of $S_i$. 
Then the multiset $\mathcal A=A_1 \cup A_2 \cup A_3 \cup A_4$ 
contains eight elements. We claim that there exists no element $x\in Y$ with multiplicity three or more in $\mathcal A$. 
To see the claim, assume for contradiction that there exists some $x \in X$ that is 
contained in three of $A_i$, $1\leq i\leq 4$. Since, for all $1 \leq i \leq 4$, the split 
$S_i|_{Y-\{x\}}$ is a 2-split of $Y-\{x\}$ if and only if $x\not\in A_i$ it follows that
 $\mathcal S|_{Y-\{x\}}$ contains at most one 2-split. But this is not possible because $\mathcal S$ 4-dices $X$
 and $|X|\geq 5$ thereby concluding the proof of the claim.
Hence, every element of $Y$ has multiplicity at most two in $\mathcal A$. Since $Y$ contains five 
elements and $\mathcal A$ has size eight, one of the following two cases must hold: 
(a) three elements of $Y$ have multiplicity two 
in $\mathcal A$ and the other two have multiplicity one and (b) 
four elements of $Y$ have multiplicity two in $\mathcal A$ and one does not appear in $\mathcal A$.

Suppose first that Case~(a) holds. Let $x$ and $y$ be the two elements in $\mathcal A$ that appear
only once. Then there exists an element $q\in Y- \{x,y\}$
such that neither $\{x,q\}$ nor $\{y,q\}$ is contained in $\{A_1,A_2,A_3,A_4\}$. 
Since $q$ has multiplicity two in $\mathcal A$ while $x$ and $y$ have 
multiplicity one each, this implies that there exist $i,j\in\{1,\ldots, 4\}$ distinct such that
the two sets $A_i$ and $A_j$ not containing 
$q$ satisfy $A_i \cup A_j=\{x,y,z,p\}$. It
follows that $\mathcal S|_{A_i\cup A_j}$ only contains the split $A_i|A_j$, 
contradicting the fact that $\mathcal S$ 4-dices $X$.

Suppose now that Case~(b) holds. Let $x$ be the element of $Y$ 
not present in $\mathcal A$. Since each element of $\{y,z,p,q\}$ appears 
twice in $\mathcal A$, it follows that, up to potentially having to relabel  the elements 
of $Y-\{x\}$, $S|_Y=\{yp|xzq, yq|xzp, zp|xyq, zq|xyp\}$ We can now use 
Lemma~\ref{cases}(ii) to conclude that $\phi_A=\phi_B$, which contradicts our assumption
that $\phi_A\not=\phi_B$.
\end{proof}

Note that the assumption that $\mathcal S$ 4-dices $X$ is necessary
for the characterization in Proposition~\ref{prop45d}~(ii) to hold.
For example, the
split system on $X=\{1, \ldots 5\}$ whose set of non-trivial splits equals
\[\{12|345, 23|451, 34|512, 45|123\}\]
does not 5-dice $X$ but $\phi_A \neq \phi_B$ 
holds for all $A,B \in {X \choose 3}$ with $|A \cap B|=1$.

We now show that injectivity of a split system can be characterized by 
considering at most 6-points.

\begin{theorem}\label{dices}
	Suppose $\mathcal S$ is a split system on $X$, $|X| \geq 3$.
	Then $\mathcal S$ is injective if and only if $\mathcal S$ 4-, 5- and 6-dices $X$.
\end{theorem}

\begin{proof}
If $|X|=3$, then the equivalence trivially holds. Hence, we may assume for the following that $|X| \geq 4$.

Assume first that $\mathcal S$ 4-, 5- and 6- dices $X$, and let $A,B \in {X \choose 3}$ distinct. 
If $|X|=4$, then $|A \cap B|=2$. In that case, Proposition~\ref{prop45d}(i) implies that 
$\phi_A \neq \phi_B$. If $|X|=5$, then $A \cap B \neq \emptyset$. In that case, 
Proposition~\ref{prop45d}(ii) implies that $\phi_A \neq \phi_B$. Finally, suppose that $|X| \geq 6$. In view of 
Proposition~\ref{prop45d}(ii), we have that $\phi_A \neq \phi_B$ 
holds in case $A \cap B \neq \emptyset$. 
It remains to show that $\phi_A \neq \phi_B$ also holds when $A \cap B=\emptyset$. To see this, let $A=\{x,y,z\}$ and $B=\{t,u,v\}$ be 
subsets of $X$ such that $A\cap B=\emptyset$. 
Let $Y=A \cup B$. Since $|X|\geq 6$ and $\mathcal S$ 6-dices $X$, the split system $\mathcal S|_Y$ 
contains either a 3-split or a triangle of 2-splits. In both cases, we can use 
Lemma~\ref{cases}(iii) to conclude that $\phi_A \neq \phi_B$.

Conversely, assume that $\mathcal S$ is injective. Then $\phi_A \neq \phi_B$ 
for all distinct $A, B \in {X \choose 3}$  with $A \cap B \neq \emptyset$. By 
Proposition~\ref{prop45d}(ii), it follows that $\mathcal S$ 4-dices and 
5-dices $X$. To see that $\mathcal S$ also 6-dices $X$, suppose that $|X| \geq 6$ and let 
$Y=\{x,y,z,t,u,v\}$ be a subset of $X$ of size 6.
 Since $\mathcal S$ 5-dices $X$ Lemma~\ref{lm-up} implies
 that $\mathcal S|_Y$ contains either a 3-split or at least eight 2-splits.
We claim that if $\mathcal S|_Y$ does not contain a 3-split
then $\mathcal S|_Y$ must contain a triangle of 2-splits. To see the claim, we remark first that 
if $\mathcal S|_Y$ contains ten 2-splits or more, then it must contain a triangle of 2-splits.
Employing a case analysis, we obtain that, up to potentially having to relabel the elements of $Y$, a 
split system on $Y$ containing eight 2-splits or more without containing a 
triangle of 2-splits is either (a) the split system $\mathcal S_1$ whose set of non-trivial splits is
$\{xy|\overline{xy}, xz|\overline{xz}, xt|\overline{xt}, xu|\overline{xu},
yv|\overline{yv}, zv|\overline{zv}, tv|\overline{tv}, uv|\overline{uv}\}$ or (b) 
a subset of the split system $\mathcal S_2$ whose set of non-trivial splits is
$\{xy|\overline{xy}, yz|\overline{yz}, zt|\overline{zt}, tu|\overline{tu},
uv|\overline{uv}, vx|\overline{vx}, xt|\overline{xt}, yu|\overline{yu}, zv|\overline{zv}\}$.
Since $\mathcal S_1$ does not 5-dice $X$ because $\mathcal S_1|_{Y-\{x\}}$ 
contains only four 2-splits it follows that  $\mathcal S|_Y\not=\mathcal S_1$. 
Hence, Case~(a) cannot hold. But Case~(b) cannot hold either since
if $\mathcal S|_Y$ is a subset of $\mathcal S_2$ then
Lemma~\ref{cases}(iii) implies $\phi_{uxz}= \phi_{tvy}$. But this is 
impossible because $\mathcal S|_Y$ is injective. Hence, 
$\mathcal S|_Y$ must contain a triangle of 2-splits, as claimed. Thus,
$\mathcal S$ also 6-dices $X$.
\end{proof}

As an important consequence of the last result, we see that 
injectivity of a split system is well behaved with respect to restriction:

\begin{corollary}\label{sub-inj}
	Suppose $\mathcal S$ is a split system on $X$ with $|X|\geq 3$.
	If $\mathcal S$ is injective, then $\mathcal S|_Y$ is injective, for all $Y \subseteq X$ with $|Y|\geq 3$.
\end{corollary}
\begin{proof}
Suppose that $Y \subseteq X$  with $|Y|\geq 3$ and that $\mathcal S$ is injective.
Then, by Theorem~\ref{dices}, $\mathcal S$ 4-, 5- and 6-dices $X$.
So  $\mathcal S|_Y$ 4-, 5- and 6-dices $Y$. By Theorem~\ref{dices}, it follows that
$\mathcal S|_Y$ is injective.
\end{proof}

\section{The injective dimension}\label{sec:id}

Recall that the dimension $\dim(\mathcal S)$ of a split system $\mathcal S$
is defined as the dimension of the largest hypercube in
$\mathcal B(\mathcal S)$ or, equivalently, 
the size of the largest incompatible subset of $\mathcal S$.
For $n\ge 3$, we define the {\em injective dimension} $\ID(n)$ of $n$ to be
\begin{equation}
\label{define}
\ID(n) = \min\{ \dim(\mathcal S) \colon \mathcal S \mbox{ is an injective split 
system on } \{1, \ldots, n\} \}.
\end{equation}
Note that since Theorem~\ref{circbu} implies that for all $X$ with $n=|X|\geq 4$ there
exists an injective split system  on $X$, the quantity $\ID(n)$ is well-defined.
We are interested in $\ID(n)$ since its value gives a lower bound for
the number of vertices in the Buneman graph of any injective split system on $X$.
In particular, if $\ID(n)=m$ then the Buneman graph $\mathcal B(\mathcal S)$
of any injective split system $\mathcal S$ on $X$ must contain 
an $m$-cube as a subgraph. Hence, $\mathcal B(\mathcal S)$ must contain at least $2^m$ vertices.

To be able to present some upper and lower bounds for $\ID(n)$ (Theorem~\ref{cor:ID-numbers}),
we first show that $\ID: \mathbb N_{\geq 3} \to \mathbb N$ is a monotone increasing function.

\begin{lemma}\label{thm:inj-dim}
For any two integers $n$ and $m$ with $n \geq m \ge 3$, we have $\ID(n) \geq \ID(m)$.
\end{lemma}
\begin{proof}
Let $\mathcal S$ be an injective split system on some set $X$ with $|X|=n$ 
such that $\dim(\mathcal S)=\ID(n)$. Let $Y$ be a subset of $X$ of size $m$. By
Corollary~\ref{sub-inj}, the split system $\mathcal S|_Y$ is injective, so 
$\ID(m) \leq \dim(\mathcal S|_Y)$. To see that $\dim(\mathcal S|_Y) \leq 
dim(\mathcal S)$ also holds it suffices to remark that if two splits $S$ and $S'$ in $\mathcal S$ are such that 
$S|_Y$ and $S'|_Y$ are incompatible then $S$ and $S'$ are also incompatible. 
Hence, an incompatible subset of $\mathcal S|_Y$ naturally induces an incompatible subset of $\mathcal S$ of the same size.
It follows 
that $\ID(m) \leq \dim(\mathcal S|_Y) \leq \dim(\mathcal S)=\ID(n)$, as desired.
\end{proof}

We now give upper and lower bounds for $\ID(n)$ where $n=|X|\geq 4$. As we shall 
see in the proof, the upper bound comes from the 
fact that a maximal circular split system on $X$ is injective by Theorem~\ref{circbu}(ii) 
and that in \cite{C08} it was shown that the maximum dimension of a hypercube in $\mathcal B(\mathcal S)$ 
is $\lfloor \frac{n}{2} \rfloor$. Note that 
the split system $\mathcal S$ formed by all splits of $X$ of size two or less is injective by Theorem~\ref{circbu}(i) and, by \cite{C08}, has 
dimension $n-1$. Indeed, two splits $S$ and $S'$ in $\mathcal S$ are incompatible if there exists an 
element $x \in X$ such that $x$ belongs to the small part of both $S$ and $S'$. Hence, the largest incompatible 
subsets of $\mathcal S$ are the subsets of the form $xy|\overline{xy}\,:\, y \in X-\{x\}\}$, some $x \in X$, and these 
subsets have size $n-1$.

\begin{theorem}\label{cor:ID-numbers}
	For all intergers $n \ge 4$, we have $\ID(n) \leq \lfloor \frac{n}{2} 
	\rfloor$.
	Moreover, $\ID(3)=1$, $\ID(4)=\ID(5)=2$, $\ID(6)=\ID(7)=\ID(8)=3$, and for all 
	$n \ge 9$, $\ID(n) \ge 3$.
\end{theorem}
\begin{proof}
	Let $X=\{1, \ldots, n\}$. To see that the first statement holds, 
	let $\mathcal S$ be a maximal circular split system on $X$. 
	If $n\geq 4$ then Theorem~\ref{circbu}(ii) implies that $\mathcal S$ is injective. 
	Hece, $\ID(n) \leq \dim(\mathcal S)$. By \cite{C08}, the Buneman 
	graph $\mathcal B(\mathcal S')$ of a maximal circular split system $\mathcal S'$ on $X$ contains an  
	$\lfloor \frac{n}{2} \rfloor$-cube, and all other subcubes in $\mathcal B(\mathcal S')$
	have no larger dimension. Hence, $\dim(\mathcal S')=\lfloor \frac{n}{2} \rfloor$.
	Thus, $\ID(n) \leq \lfloor \frac{n}{2} \rfloor$.
	
	To see the remainder of the theorem, note first that $\ID(3)=1$ 
	since, as was mentioned in Section~\ref{sec:c2} already, the unique split system
	on $X$ is injective and $\mathcal B(\mathcal S)$ is a phylogenetic tree on $X$
		
	To see that $\ID(4)=\ID(5)=2$ holds, we first remark that 
	in view of the first statement of the theorem, we have $\ID(4) \leq 2$ 
	and $\ID(5) \leq 2$. Now, let $X$ be such that $n \in \{4,5\}$ and 
	assume for contradiction that there exists an injective split 
	system $\mathcal S$ on $X$ with $\dim(\mathcal S)=1$. In particular, $\mathcal S$ is compatible. Then 
	$\mathcal B(\mathcal S)$ is a phylogenetic tree on $X$
	and has $|\mathcal S|+1$ vertices. Moreover, since 
	a compatible split system on $X$ has at most $2n-3$ elements (see e.g. \cite[Theorem 3.3]{D12}),	
	it follows that $\mathcal B(\mathcal S)$ has at most $2$ 
	internal vertices if $n=4$, and at most $3$ internal 
	vertices if $n=5$. But $\mathcal S$ is injective, so $\mathcal B(\mathcal S)$ 
	must have at least ${4 \choose 3}=4$ internal vertices if $n=4$, 
	and at least ${5 \choose 3}=10$ internal vertices if $n=5$, a contradiction. 
	Hence, $\ID(4)=\ID(5)=2$.
	
 We continue with showing 
	that $\ID(6)\ge 3$ from which it then follows by Lemma~\ref{thm:inj-dim} and the
	first statement of the theorem 
	that $\ID(6)=\ID(7)=3$ and that $\ID(n) \ge 3$, for all $n \ge 8$.
	Suppose that $\mathcal S$ is an injective split system on $X=\{1,\ldots, 6\}$.
	Bearing in mind that, by Theorem~\ref{dices}, $\mathcal S$ 4-, 5- and 6-dices $X$
	we next perform a case analysis on the number of 3-splits in $\mathcal S$.
	If $\mathcal S$ contains three 3-splits or more then $\dim(\mathcal S) \geq 3$ since 
	all 3-splits of $X$ are pairwise incompatible.
	
	If $\mathcal S$ contains two 3-splits, say $123|456$ and $234|561$, 
	then since $\mathcal S$ 4-dices $X$ it follows that there must exist  a split $S\in\mathcal S$ 
	such that $S(2) \neq S(3)$ and $S(5) \neq S(6)$. Since the splits $S$, $123|456$, and $234|561$ are pairwise 
	incompatible, we obtain $\dim(\mathcal S) \geq 3$.
	
	If $\mathcal S$ contains one 3-split, say $123|456$, then one of the following two cases must hold. 
	If there exists an element $x \in X$ and three splits $S_1$, $S_2$, and $S_3$ in $\mathcal S$ containing $x$ 
	in their small part then 
	$\dim(\mathcal S) \geq 3$ because $\{S_1,S_2,S_3\}$ is incompatible. If no such element $x$ exists then 
	$\mathcal S$ contains at most six 2-split. An exhaustive search shows that, up to potentially having to relabel the elements
	in $\{1,2,3\}$, there exists only one such split system that is injective 		i.\,e.\,$\mathcal S$ is the split system
	whose subset of non-trivial splits is the set
	\[
	\{123|456, 15|\overline{15}, 16|\overline{16}, 24|\overline{24}, 26|\overline{26}, 34|\overline{34}, 35|\overline{35}\}.
	\]
	One can then easily verify that $\{123|456,15|\overline{15}, 16|\overline{16}\}$ is incompatible. Hence, $\dim(\mathcal S)\geq 3$ in this case.
	
	Finally, if $\mathcal S$ does not contain a 3-split, then 
	it must contain a triangle of 2-splits because $\mathcal S$ 6-dices $X$. Since the  three splits in such a triangle 
	are pairwise incompatible it follows that $\dim(\mathcal S) \geq 3$. This concludes the proof that $\ID(6)\geq 3$.
	
	To show that $\ID(8) =3$, we employed Theorem~\ref{dices} and used a computer program to verify
	that $\mathcal S$ is the split system whose subset of non-trivial splits is
	\[
	\begin{array}{r l}
	&\{1234|5678,1357|2468,123|\overline{123},246|\overline{246},478|\overline{478},156|\overline{156},12|\overline{12}, 34|\overline{34},56|\overline{56},78|\overline{78},26|\overline{26},\\
	&35|\overline{35},17|\overline{17},48|\overline{48},68|\overline{68},57|\overline{57},23|\overline{23}\}
	\end{array}
	\]
	is injective. Since $\dim(\mathcal S)=3$, it follows that 
	$\ID(8)=3$.
\end{proof}

Note that as $\ID(8)=3$, the upper bound for $\ID(n)$ given
in Theorem~\ref{cor:ID-numbers} is not 
tight even for $n=8$.
In general, it appears to be difficult to find a better upper or lower bounds for 
$\ID(n)$, however in 
the next two sections we shall give improved bounds for two variants of the injective dimension.

\section{The injective 2-split-dimension}\label{sec:id2}

To help better understand the injective dimension of a split system, 
in this section we shall consider a restricted version of this
quantity that is defined as follows.
For $n\geq 3$, let $\mathbb{S}_2(n)$ be the set of all injective split systems
on $X=\{1,\dots, n\}$ whose non-trivial splits all have size 2.
As mentioned in the introduction, we define $\ID_2(n)$ for $n\ge 3$ as
\begin{equation}
\label{eq:define-2}
\ID_2(n) = \min\{ dim(\mathcal S) \colon \mathcal S\in  \mathbb{S}_2(n) \}.
\end{equation}

By Theorem~\ref{circbu}~(i), $\ID_2(n)$ is well-defined. 
Clearly $\ID_2(n) \geq \ID(n)$ and 
equality holds for $n =3,4,5$ since every non-trivial split 
of a set $X$ of size 3, 4, or 5 is a 2-split.
In the main result of this section (Theorem~\ref{id2-lb}), we provide upper and lower
bounds for $\ID_2(n)$. To prove it, we shall use two lemmas.

For $\mathcal S$ a split system on $X$, we denote by $P(\mathcal S)$ the graph 
with vertex set $X$ and with edge set all the pairs $\{x,y\}$ such that $xy|\overline{xy} \in \mathcal S$. 
We also denote the degree of a vertex $x \in X$ in $P(\mathcal S)$ by $\deg_{P(\mathcal S)}(x)$.
If $\mathcal S$ contains only trivial splits and 2-splits then $P(\mathcal S)$ and dicing are related as stated
as in Lemma~\ref{dice2}. We omit its straight-forward proof but remark in passing that Lemma~\ref{dice2}
is a strengthening of  Theorem~\ref{dices} for split systems in $\mathbb{S}_2(n)$, for all $n\geq 3$.

\begin{lemma}\label{dice2}
Let $\mathcal S \in \mathbb{S}_2(|X|)$ be a split system on $X$ with $|X| \ge 3$.
Then, 
\begin{itemize}
\item[$\bullet$] $\mathcal S$ 4-dices $X$ if and only if $|X| \leq 4$ or for all $Y \in {X \choose 4}$, 
the restriction $P(\mathcal S|_Y)$ contains two edges that share a vertex.
\item[$\bullet$] $\mathcal S$ 5-dices $X$ if and only if $|X| \leq 5$ or for all $Y \in {X \choose 5}$, 
 the restriction $P(\mathcal S|_Y)$ contains five edges or more.
\item[$\bullet$] $\mathcal S$ 6-dices $X$ if and only if $|X| \leq 6$ or for all $Y \in {X \choose 6}$, 
the restriction $P(\mathcal S|_Y)$ contains a 3-clique.
\end{itemize}
\end{lemma}

In terms of the dimension of a split system in $\mathbb{S}_2(n)$, $n\geq 3$, we also have
the following result.

\begin{lemma}\label{dim2}
Let $\mathcal S  \in \mathbb{S}_2(|X|)$ be a split system on $X$ with $|X| \ge 3$.
Then, 
\begin{itemize}
\item[(i)] If $P(\mathcal S)$ does not contain a 3-clique then 
$\dim(\mathcal S)=\max_{x \in X} \{\deg_{P(\mathcal S)}(x)\}$.
\item[(ii)] If $P(\mathcal S)$ contains a 3-clique 
then $\dim(\mathcal S)=\max \{\max_{x \in X} \{\deg_{P(\mathcal S)}(x)\}, 3\}$.
\end{itemize}
\end{lemma}

\begin{proof}
We prove (i) and (ii) together. For this, put $n=|X|$. If $n=3$ then $P(\mathcal S)$ consists of three isolated vertices. So Assertion~(i) 
holds. Since $\mathcal S$ only contains trivial splits,  it follows that Assertion~(ii) holds vacuously. So assume 
that $n\geq 4$.
Let $\mathcal S \in \mathbb{S}_2(n)$.
Then a maximal incompatible subset $\mathcal S'$ of 
$\mathcal S$ must be of one of the following two types:
\begin{itemize}
\item[(a)] A triangle of 2-splits.
\item[(b)] The set of all 2-splits in $\mathcal S$ containing some $x\in X$ in their small part.
\end{itemize}

To see that these are the only two possible types, it suffices to remark that any subset $\mathcal S'\subseteq \mathcal S$ with 
$|\mathcal S'|\geq 4$ is incompatible if and only if there exists some $x\in X$ such that all splits of $\mathcal S'$ contain  $x$ in their small part.

If $\mathcal S'$ is of Type~(a) then $\mathcal S'$ corresponds to a 3-clique in  $P(\mathcal S)$
and $|\mathcal S'|=3$. If $\mathcal S'$ is of Type~(b) then  $\mathcal S'$ corresponds to the 
set of edges of $P(\mathcal S)$ that are incident with $x$.  Hence, $|\mathcal S'|=\deg_{P(\mathcal S)}(x)\geq 3$. Thus, 
if $P(\mathcal S)$ has a vertex $x$ with $\deg_{P(\mathcal S)}(x)\geq 3$ or if $P(\mathcal S)$ does 
not contain a 3-clique then $\dim(\mathcal S)=\max_{x \in X} \{\deg_{P(\mathcal S)}(x)\}$. 
Otherwise, $\dim(\mathcal S)=3$.
\end{proof}

We now prove the main result of this section.

\begin{theorem}\label{id2-lb}
For all $n \geq 5$,  	
\[
\lfloor \frac{n}{2} \rfloor  \le \ID_2(n) \leq n-3.
\]
\end{theorem}

\begin{proof}
We first show that $\ID_2(n) \leq n-3$ by constructing an injective 
split system $\mathcal S_n$ on $X_n=\{1, \ldots, n\}$ with $\dim(\mathcal S)=n-3$.
For this, let $\sigma_n$ denote some circular ordering of the 
elements of $X_n$. Let $\mathcal S_n$ denote the set of all splits $xy|X_n-\{x,y\}$ 
such that $x,y \in X_n$ are not consecutive under $\sigma_n$. By definition of $\mathcal S_n$, all 
vertices of $P(\mathcal S_n)$ have degree $n-3$. If $n \geq 6$, it follows by 
Lemma~\ref{dim2} that $\dim(\mathcal S_n)=n-3$. If $n=5$, it is straight-forward to 
check that $P(\mathcal S_n)$ does not contain a 3-clique. So, by 
Lemma~\ref{dim2}, $\dim(\mathcal S_n)=n-3$ holds in this case too. 
Thus, it remains to show that $\mathcal S_n$ is injective. In view of Theorem~\ref{dices}, 
we do this by showing that $\mathcal S_n$ 4-, 5- and 6-dices $X_n$.

To see that $\mathcal S_n$ 4-dices $X_n$, let
$Y \in {X_n \choose 4}$ which exists as $n\geq 5$. By Lemma~\ref{dice2}, it suffices to show that 
there exists an element of $Y$ that has degree 2 or more in $P(\mathcal S_n|_Y)$. Let $x \in Y$. 
If $\deg_{P(\mathcal S_n|_Y)}(x)\geq 2$, we are done by the definition of $\mathcal S_n$. 
Otherwise, $Y$ contains two 
elements $y$ and $z$ such that $y$ and $z$ precede and follow $x$ 
under $\sigma_n$, respectively. Let $t$ be the fourth element of $Y$. Then $\{x,t\}$ is
 an edge in $P(\mathcal S_n)$. Moreover, since $n \geq 5$ and $t \neq x$, there 
must be at least one of $y,z$ that is adjacent with $t$ in $P(\mathcal S_n)$. Thus, 
$\deg_{P(\mathcal S_n|_Y)}(t)\geq 2$, as required.

To see that $\mathcal S_n$ 5-dices $X_n$, let
$Y \in {X_n \choose 5}$ which again exists because $n\geq 5$. By Lemma~\ref{dice2}, it suffices to show that 
$P(\mathcal S_n|_Y)$ contains at least five edges. To see this, note first that, for all $x \in X$, there are at 
most two elements in $Y-\{x\}$ that do not form an edge with $x$ in $P(\mathcal S_n|_Y)$ 
because $\mathcal S_n$ is circular. For all
 $x \in Y$, it follows that $\deg_{P(\mathcal S_n|_Y)}(x)\geq 2$. Since $Y$ 
contains five elements, this imples that $P(\mathcal S_n|_Y)$ contains at least five edges, as required.

Finally, to see that $\mathcal S_n$ 6-dices $X_n$, note first that we
may assume that $|X|\geq 6$ as otherwise $\mathcal S_n$ 6-dices $X_n$ by definition.
Let $Y \in {X_n \choose 6}$. By Lemma~\ref{dice2}, it suffices to show 
that $P(\mathcal S_n|_Y)$ contains a 3-clique. To see this, let $x \in Y$. 
Then, by the definition of $P(\mathcal S_n)$, there exist at least 
three elements in $Y$, say $y$, $z$ and $t$, that form an edge with $x$ in $P(\mathcal S_n)$. 
Moreover, at least two of $y$, $z$ and $t$, say $y$ and $z$, must form an 
edge $\{y,z\}$ in $P(\mathcal S_n)$ since $y$, $z$ and $t$ cannot all be 
consecutive with each other under $\sigma_n$. It follows that
 $\{x,y,z\}$ is the vertex set of a 3-clique in $P(\mathcal S_n|_Y)$, as required.
This concludes the proof that $\ID_2(n) \leq n-3$.

We now show that $\lfloor \frac{n}{2} \rfloor  \le \ID_2(n)$.
We begin by showing that $\ID_2(n+2)>\ID_2(n)$, for all $n \geq 3$. 
Assume that $n\geq 3$. Also, assume that $\sigma_{n+2}$ is the natural
ordering of $X_{n+2}=\{1,2,\ldots, n,n+1, n+2\}$. Let $\mathcal S\in \mathbb{S}_2(n+2)$ denote a split system on
$X_{n+2}$ that attains $\ID_2(n+2)$.
Let $\mathcal S'$ denote a maximal incompatible subset of $\mathcal S$. 
We claim that $\mathcal S'$ must contain a non-trivial split that separates 
the elements $n+1$ and $n+2$. Clearly, $\mathcal S$ must contain such a split as
otherwise Lemma~\ref{medinbu} implies that $\phi_Y(S)=\phi_{Y'}(S)$ holds for all $S \in \mathcal S$ and 
all $Y,Y' \in {X_{n+2} \choose 3}$ with $Y \cap Y'=\{n+1,n+2\}$. Hence, $\mathcal S$ is not injective
which is impossible.
Choose a split $S_0\in\mathcal S$ such that $S_0(n+1)\not=S_0(n+2)$. Assume for contradiction that all splits $S \in \mathcal S'$ 
satisfy $S(n+1)=S(n+2)$. Then $\mathcal S_0$ is incompatible with every split
in $\mathcal S'$ because $S_0$ and every split in $\mathcal S'$ have size two.
Hence, $\mathcal S'\cup\{S_0\}$ is an incompatible subset of $\mathcal S$ that contains $\mathcal S'$ as a proper
subset which contradicts the choice of $S'$.

Consider now the restriction $\mathcal S_n$ of $\mathcal S$ to $X_n$. By 
Corollary~\ref{sub-inj}, $\mathcal S_n$ is injective because $\mathcal S$ is injective. Moreover, since all maximal 
incompatible subsets of $\mathcal S$ contain a split separating $n+1$ and $n+2$ by the previous claim, it follows that no maximal incompatible subset of $\mathcal S_n$ has size equal to
$\dim(\mathcal S)$. Hence, $\dim(\mathcal S_n)< \dim(\mathcal S)$. 
Since $\dim(\mathcal S)=\ID_2(n+2)$ by the choice of $\mathcal S$, and 
$\dim(\mathcal S_n) \geq \ID_2(n)$ by the injectivity of $\mathcal S_n$, it follows 
that $\ID_2(n+2)>\ID_2(n)$, as required.

We conclude with showing that  $\ID_2(n) \geq \lfloor \frac{n}{2} \rfloor$ holds by performing
induction on $n$. 
If $n =5$ then $\ID_2(n)=\ID(n)$ since all non-trivial splits on $X_n$ 
are 2-splits and $\ID(n)= \lfloor \frac{n}{2} \rfloor$ holds by 
Corollary~\ref{cor:ID-numbers}. This implies the stated inequality  
in this case. Now, let $n>5$ and assume that the stated 
inequality holds for all $5\leq n'<n$. Since $\ID_2(n)>\ID_2(n-2)$
it follows by induction hypothesis 
that $\ID_2(n)> \ID_2(n-2)\geq \lfloor \frac{n-2}{2} \rfloor$. Hence,
 $\ID_2(n) \geq \lfloor \frac{n-2}{2} \rfloor +1=\lfloor \frac{n}{2} \rfloor$, 
as desired.
\end{proof}

\section{Rooted injective dimension}\label{sec:idr}

In this section, we consider another variant of the injective
dimension which behaves quite differently from $\ID(n)$.
Let $X$ denote a set with $|X|=n$. Choose some element $r \in X$. 
For $Z \in {X-\{r\} \choose 2}$, put $Z_r=Z \cup \{r\}$.
We say that a split system is {\em rooted-injective (relative to $r$)} if
\[
\phi_{Z_r} \neq \phi_{Z'_r}
\]
for all $Z,Z' \in {X \choose 2}$ distinct. This 
concept is closely related to the rooted median graphs considered in \cite{B22}.
Note that if $X=3$ then the (unique) split system on $X$ is $r$-rooted injective for
any choice of $r\in X$. Also,
note that if $\mathcal S$ is injective, then $\mathcal S$ is 
rooted-injective relative to $r$, for all $r \in X$. The converse, however, 
does not hold. For example, the split system $\mathcal S$ on $X=\{1, \ldots, 6\}$ 
whose set of non-trivial splits is:
\[
\{14|\overline{14}, 15|\overline{15}, 16|\overline{16}, 24|\overline{24}, 25|\overline{25}, 26|\overline{26}, 
34|\overline{34}, 35|\overline{35}, 36|\overline{36}\}
\]
is not injective because $\mathcal S$ does not 6-dice  $X$ and so Theorem~\ref{dices} does not hold. But $\mathcal S$
is rooted-injective relative to $r$, for all $r \in X$.

For $n \ge 3$, $X$ a set with $|X|=n$ and some $r\in X$, we define the
{\em rooted-injective dimension} $\ID^r(n)$ to be
\[
\ID^r(n) = \min\{  dim(\mathcal S) : \mathcal S \mbox{ is a rooted-injective 
split system on $X$ relative to $r$} \}.
\]

Our next result (Theorem~\ref{th:idr}) shows that $\ID^r(n)$ is well-defined for all $n\geq 3$, 
and that, in contrast to $\ID_2(n)$,  $\ID^r(n)$ is always equal to 2 when $n \geq 4$.

\begin{theorem}\label{th:idr}
	Suppose that $X$ is such that  $n=|X|\ge 4$ and that $r \in X$. Then 
	there exists a rooted-injective split system $\mathcal S$ on $X$ relative to $r$ with $\dim(\mathcal S)=2$.
	Moreover $\ID^r(n) =2$.
\end{theorem}
\begin{proof}
Put $X = \{1,2,\dots,n-1,r\}$. 
First note that $\ID^r(n)\ge 2$, since if $\ID^r(n) =1$, 
then there would be a rooted-injective split system $\mathcal S$ 
on $X$ relative to $r$ with $\dim(\mathcal S)=1$. But this is not possible 
since then the Buneman graph $B(\mathcal S)$ associated to $\mathcal S$ would be
a phylogenetic tree on $X$ with $|\mathcal S|+1$ edges. Using a similar argument to the one used 
to show that $\ID(4)=\ID(5)=2$ in the proof of Theorem~\ref{cor:ID-numbers}, it is straight-forward to
check that then $\mathcal S$ is not rooted-injective which is impossible.

Now, define the split system $\mathcal S$ on 
$X$ whose subset of non-trivial splits is equal to $\mathcal S_1 \cup \mathcal S_2$, where:
\[
\mathcal S_1=\{ \{n-1-i,\ldots,n-1\} | \overline{\{n-1-i,\ldots,n-1\}} \cup\{r\} \,:\, 0 \le i \le n-3\}
\]
and
\[
\mathcal S_2=\{ \{n-1-i,\ldots,1\} | \overline{\{n-1-i,\ldots,1\}} \cup\{r\} \,:\, 0 \le i \le n-3\}. 
\]

\begin{figure}[h]
\begin{center}
\includegraphics[scale=0.7]{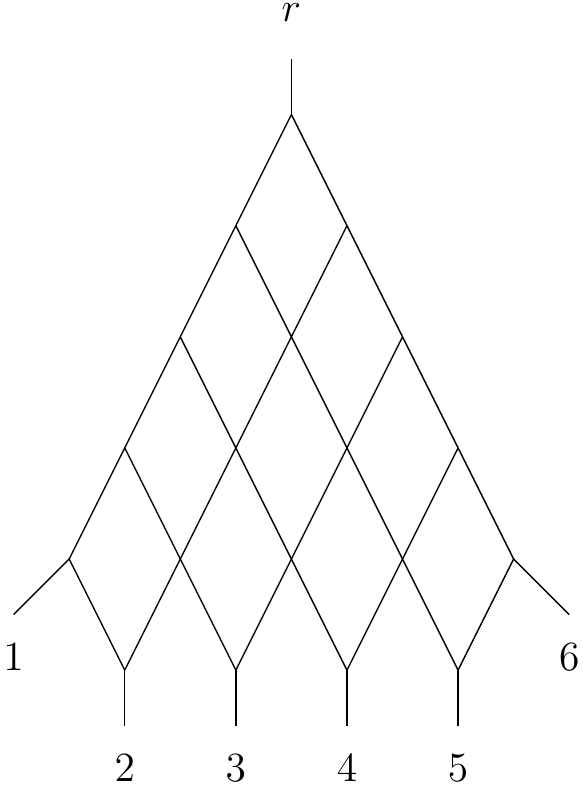}
\caption{The Buneman graph of a split system on $\{1,2,3,4,5, 6,r\}$ that is rooted injective relative to $r$
that is constructed as described in the proof of Theorem~\ref{th:idr}.}
\label{fig:grid}
\end{center}
\end{figure}

For example, for $n=7$, the Buneman graph $B(\mathcal S)$ of $\mathcal S$ 
is the half-grid pictured in Figure~\ref{fig:grid}. More precisely, in that figure, 
the splits in $\mathcal S_1$ and $\mathcal S_2$ are the splits associated to edges 
oriented downwards from left to right and from right to left, respectively. 

To see that $\dim(\mathcal S)=2$, it suffices to remark that 
$\mathcal S_1$ and $\mathcal S_2$ are compatible, so a maximal incompatible 
subset of $\mathcal S$ has size at most $2$. Since $\mathcal S$ is not compatible, 
it follows that $\dim(\mathcal S)=2$.

We next show that $\mathcal S$ is rooted-injective relative to $r$. To see this, 
let $Z, Z' \in {X-\{r\} \choose 2}$ distinct. Also, let $x^-=\min(Z \cup Z')$ 
and $x^+=\max(Z \cup Z')$. Since the $Z \cup Z'$ has size 
at least $3$, we have that $x^-$ and $x^+$ are distinct. Furthermore, $x^- \leq n-3$ 
and $x^+ \geq 3$ must hold.  In particular, the splits $S^-=\{x^-+1, \ldots, n-1\}|\{1, \ldots, x^-\} \cup \{r\}$ 
and $S^+=\{1, \ldots, x^+-1\}|\{x^+, \ldots, n-1\} \cup \{r\}$ belong to $\mathcal S_1$ 
and $\mathcal S_2$ respectively, so both splits belong to $\mathcal S$. 
Moreover, $Z \cap Z'$ contains at most one element, so at least one of $x^-$ and $x^+$ 
does not belong to $Z \cap Z'$. If $x^- \notin Z \cap Z'$ 
then $S^-$ satisfies $\phi_{Z_r}(S^-) \neq \phi_{Z'_r}(S^-)$, and if  if $x^+ \notin Z \cap Z'$
then $S^+$ satisfies $\phi_{Z_r}(S^+) \neq \phi_{Z'_r}(S^+)$. So, $\mathcal S$ is rooted-injective relative to $r$.
\end{proof}

\begin{remark}
	The proof that the split system $\mathcal S$ is rooted-injective relative to $r$ in
	Theorem~\ref{th:idr} gives an alternative 
	proof that the extended half-grid for $(n+1)$ in \cite[p.7]{B22} can be used to 
	represent a symbolic map,
	since the Buneman graph  $B(\mathcal S)$ with 
	the pendant edge containing $r$ contracted is isomorphic to 
	the extended half-grid on $n$.
\end{remark}

Note that the rooted-injective split system $\mathcal S$ in the proof of Theorem~\ref{th:idr}
is the union of two split systems $\mathcal S_1$ and $\mathcal S_2$
whose associated Buneman graphs are phylogenetic trees.
In general, if $\mathcal S$ is a split system on $X$ with this property
then $\dim(\mathcal S)\leq 2$ (since every 
3-subset of $\mathcal S$ must contain at least one pair of splits 
that is contained in one of the split systems, and so this pair of splits must be compatible).
Hence, by Theorem~\ref{cor:ID-numbers},  $\mathcal S$ cannot be injective in case $|X| \ge 6$.

\section{Discussion}\label{sec:discuss} 

In this paper we have defined and explored the concept of injective split
systems, that is, splits systems $\mathcal S$ on a set $X$ such that two
distinct sets of three elements of $X$ have distinct median vertex in the
Buneman graph $\mathcal B(\mathcal S)$ associated to $\mathcal
S$. Making use of the notion of dicing, we have shown that a given
split system is injective if and only if its subsets of size $6$ or less
are injective, from which we derived a characterization of injective split
systems. We also studied the injective dimension of an integer $n \geq 3$,
that is, the minimal dimension of an injective split system on some set of
$n$ elements. On this topic, it remains an open question whether there is a
lower bound for $\ID(n)$ that is linear in $n$.

The notion of an injective split system also suggests to consider a
matching concept of surjective split systems.  We call a split system $\cS$
on some set $X$ with $|X|\geq 3$ \emph{surjective} if the vertex set of
$B(\mathcal S)$ is equal to
\begin{equation}
  \{\phi_x \,:\, x \in X\} \cup \{ \phi_Y \,:\, Y \in {X \choose 3}\},
\end{equation}
In other words, every non-leaf vertex in $B(\cS)$ is the median of three
leaves in $B(\cS)$. Note that every split system whose Buneman graph is a
phylogenetic tree is surjective but, for example, the split system
corresponding to the Buneman graph in example in Fig.~\ref{bubu}(ii) is not
surjective because the central vertex in the graph is not the median of any
three leaves. The general properties of surjective split systems remain to
be investigated.

Naturally, one may want to study \emph{bijective} split system $\cS$
 that are both injective and surjective. We conjecture that a split
system $\cS$ on some set $X$ with $|X|\geq 3$ is bijective if and only if
either $|X|=3$ and $|\cS|=3$ or $|X|=4$, $|\cS|=6$ (i.\,e.\,the Buneman
graph associated to $\cS$ is a three-leaved phylogenetic tree or -- up to
leaf relabelling -- the graph in Fig.~\ref{bubu}(i), respectively).  A
proof or counter-example for this conjecture might use concepts that are
related to the so-called median stabilization degree of a median algebra --
see e.g. \cite{B99,evans1982median}.
  
Finally, another interesting open problem is the following: Can we develop
a modular decomposition theory for Buneman graphs along the lines described
in \cite{B22}?

\begin{acknowledgements}
  This work was supported in part by the German Federal Ministry for
    Education and Research (BMBF 031L0164C, RNAProNet, to P.F.S.). The authors would like to thank the Institut Mittag-Leffler in Djursholm, Sweden for hosting the conference ``Emerging Mathematical Frontiers in Molecular Evolution" in August 2022, where this work was finalized.
\end{acknowledgements}

\section*{Conflict of interest}

The authors declare that they have no conflict of interest.

\section*{Data availability}

Not applicable.

\bibliographystyle{spmpsci}
\bibliography{injective}

\end{document}